\documentclass[onecolumn,11pt]{IEEEtran} 
\renewcommand{\baselinestretch}{1.5}

\IEEEoverridecommandlockouts                              
\usepackage{graphics} 
\usepackage{epsfig} 
\usepackage{amsmath} 
\usepackage{amssymb,bm}  
\usepackage{color} 

\usepackage[ruled]{algorithm2e}
\usepackage{amsmath}

\newtheorem{theorem}{Theorem}

\newtheorem{lemma}{Lemma}
\newtheorem{definition}{Definition}

\newtheorem{corollary}{Corollary}

\title{\LARGE \bf
Minimum Number of Probes for Brain Dynamics Observability
}

\author{S\'ergio Pequito $^{\star}$ $\quad$ Paul Bogdan  $^{\diamond}$  $\quad$ George J. Pappas  $^{\star}$  
\thanks{This work was supported in part by the TerraSwarm Research Center,
one of six centers supported by the STARnet phase of the Focus Center
Research Program (FCRP) a Semiconductor Research Corporation program
sponsored by MARCO and DARPA. P.B. acknowledges the support by NSF 1453860 and 1331610 grants.}
\thanks{
$^{\star}$ Department of Electrical and Systems Engineering, School of Engineering and Applied Science, University of Pennsylvania}
\thanks{
$^{\diamond}$ USC Ming Hsieh Department of Electrical Engineering, Viterbi School of Engineering, University of Southern California}
}
\begin{document}

\maketitle

\begin{abstract}
In this paper, we address the problem of placing sensor probes in the brain such that the system dynamics' are generically observable. The system dynamics whose states can encode for instance the fire-rating of the neurons or their ensemble following a neural-topological (structural) approach, and the sensors are assumed to be dedicated, i.e., can only measure a state at each time. Even though the mathematical description of brain dynamics is (yet) to be discovered, we build on its observed fractal characteristics and assume that the model of the brain activity satisfies  fractional-order  dynamics. 

Although the sensor placement explored in this paper is particularly considering the observability of brain dynamics, the proposed methodology applies to any fractional-order linear system. Thus, the main contribution of this  paper is to  show how to place the minimum number of dedicated sensors, i.e., sensors measuring only a state variable,  to ensure generic observability in discrete-time fractional-order systems for a specified finite interval of time. Finally, an illustrative example of the main results is provided using  electroencephalogram (EEG) data.
\end{abstract}

\section{INTRODUCTION}

In the recent years, control  theoreticians have resorted  to  the tools available to study, analyze and design dynamical systems to better understand the dynamical behavior of the brain~\cite{NeuralNetControllingNodes,controlBrainNet2014,FrancisModeling,FranciSIAM,RuthTAC13,Franci2013,Drion}. Often, the dynamic model of the brain  assumes a behavior at the neuron level, that is then analyzed in interaction with different neurons to which it communicates with~\cite{FrancisModeling,controlBrainNet2014}. A commonly  used model for the neuron is the Hodgkin and Huxley, but several alternatives have been proposed.  Most of the models have the following similarity: a neuron \emph{fires} (i.e., beams a signal to neighboring neurons) if a certain \emph{threshold} (or, \emph{gain}) on its \emph{state} (or \emph{stimulus}) is reached. Nonetheless, the evolution of the state is believed to be  non-linear in general, with a time-varying  threshold. Therefore, several experimental  efforts concluded that  several regions in the brain exhibit specific properties. In particular, the non-linear dynamics in certain regions can be sequentially approximated by linear time-invariant systems, leading to a linear time-varying systems that can then be restricted to a finite possible dynamics due to some plasticity and invariant properties of the brain regions. Alternatively, the response to a stimulus  may exhibit a long-range memory behavior  which may indicate the presence of a fractional-order dynamics; for instance, the vesicular-ocular system, the fly motion sensitive neuron and neocortical neurons~\cite{Lundstrom}, just to name a few. 
Nonetheless, several other dynamical systems are described by similar dynamics and  sensing capabilities. For instance,  fractional-order dynamics include artificial pancreas~\cite{Ghorbani}, on-chip traffic regulation~\cite{B2012}, traffic dynamics across road arteries~\cite{B2011}, electrical and thermal machines, transmission and acoustics ~\cite{fracOrderdiscreteJournal}, just to name a few.  

Hereafter, instead of determining a specific model to the brain's dynamics given the neurons activity (or an ensemble of these) and their interactions, to which several studies have been conducted~\cite{Kaslik,Zhou2008973}, we assume that the brain dynamics can be described by  fractional-order dynamics. Then, we aim to provide a methodology to determine the minimum placement of \emph{dedicated} sensors, i.e., sensors that  measure a single state of the system, ensuring that  the system is  \emph{generically observable}; more precisely, by generically observable (also known as structurally observable -- see Section~\ref{prelim} for formal definition),  we mean that almost all parametrizations (that are unknown) for a given structure of interaction (that can be known or imposed by neural-topological analysis). The reason we focus in  determining such placement in the context of brain observability dynamics is threefold: (i) some probes (sensors) are invasive, i.e., require surgery to deploy the probes in the brain, which may be life threatening; (ii) nowadays deployment of probes in the brain neglects  entirely the brain dynamics; and (iii) data collected in these probes is commonly used to feedback into the system, therefore, a smaller collection of data leads to a faster response time with potential implications in neurodegenerative diseases such as Parkinson's and epilepsy~\cite{SpectrumFev15}.

To the best of the authors knowledge there are no systematic methods to do probe (sensor) placement in the brain. Instead, researchers have focused on collecting data from different regions of the brain and/or developing new sensors leveraging some of the new materials (some biodegradable that do not require the extraction after a period of time). On the other hand, if we consider dynamical systems, most sensor placement methodologies are either heuristic or greedy due to the combinatorial nature of the problem, see~\cite{PequitoJournal} and references therein. Structural systems theory studies the coupling between state variables rather than the specific strength of the coupling, suitable to the problems explored hereafter, since these are not available neither measurable directly, therefore they need to be estimated after data has been collected. The structural treatment of dynamical properties of switching systems and, in particular, the obtention of sufficient and necessary conditions to ensure \emph{structural observability}, i.e., structural observable for some non-empty window of time, have been explored in \cite{Hihi09,Liu20133531}. This approach contrasts with the one presented in~\cite{allerton}, where conditions to ensure  structural observability for all non-empty windows of time. In~\cite{allerton}, the deployment of dedicated sensors was explored, but  the sensor placement to ensure structural observability for some non-empty window of time, has not yet been addressed. The present work also extends~\cite{PequitoJournal,PequitoACC}, where the minimum placement of dedicated sensors to ensure structural observability was addressed for linear time-invariant systems. \hfill $\circ$

\textbf{Main Contribution: }  The main contributions of the paper is to  show how to place the minimum number of dedicated sensors  to ensure structural observability in discrete-time fractional-order systems for a specified finite interval of time. 

The rest of the this paper is organized as follows. First, in Section~\ref{probStatement}, we present the formal statements addressed in the present paper. We introduce and revise some preliminary concepts and results in Section \ref{prelim}. In Section~\ref{mainresults}, we present the main technical results, followed  by an illustrative example in Section \ref{illustrativeexample}. Finally, conclusions and  discussion avenues for further research are presented in Section \ref{conclusions}.


\section{Problem Statement}\label{probStatement}

In this section, we formally introduce the dedicated sensor placement for   discrete-time fractional order dynamics. Formally, this problem is described as follows.

\noindent $\mathcal P_1$  
Consider a model dynamics described by a linear discrete-time fractional-order
system as follows

\begin{equation}
\Delta^{\boldsymbol{\alpha}}x_{k+1}=\left[ \begin{array}{c}
\Delta^{\alpha_1} x^1_{k+1}\\
\Delta^{\alpha_2} x^2_{k+1}\\
\vdots\\
\Delta^{\alpha_n} x^n_{k+1}
\end{array}
\right]=\sum\limits_{j=0}^{k+1}A_j x_{k},\label{fracDyn}
\end{equation}
where $\alpha_i\in\mathbb{R}^{+}$($i=1,\ldots,n$), $A_0=A$ and $A_j=\text{diag}(${\small$- (-1)^{j+1} \binom{\alpha_1}{j+1}, - (-1)^{j+1} \binom{\alpha_2}{j+1},\ldots, - (-1)^{j+1} \binom{\alpha_n}{j+1}$}$)$, for brevity, we refer to the dynamics in~\eqref{fracDyn} as $\mathcal F(A;\boldsymbol{\alpha},K)$ which is associated with the representation of all matrices $A_j$, with  $j=0,\ldots,K$. Please note that $\alpha$ is a non-integer number and the term $\binom{\alpha}{j} = \frac{\Gamma(\alpha+1)}{\Gamma(j+1)\Gamma(\alpha-j+1)}$ is expressed via the Gamma function $\Gamma (x) = \int_{0}^{\infty} t^{x-1}e^{-t}dt$ \cite{Baleanu}.

Hereafter, given $\mathcal F(A;\boldsymbol{\alpha},K)$  we aim to determine the minimum number of dedicated sensors $\mathcal J$ required to ensure that the linear discrete-time fractional-order
system~\eqref{fracDyn} with measured output
\begin{equation}
y_k=\mathbb{I}_n^{\mathcal J}x_k,
\label{outputDiscrete}
\end{equation}
 is structurally observable, i.e., \begin{equation}
\begin{array}{cc}
\arg\min\limits_{\mathcal J\subset \{1,\ldots,n\}} & |\mathcal J|\\
\text{s.t.} & (\mathcal F(A;\boldsymbol{\alpha},K),\mathbb{I}_n^{\mathcal J}) \text{ is structurally observable,}
\end{array}
\label{optProbP2}
\end{equation}
where $\mathcal J\subseteq\{1,\ldots,n\}$ denotes the set of indices of the  dedicated sensors,  $\mathbb{I}_n^{\mathcal J}$ represents the subset of columns of the identity matrix $\mathbb{I}_n$ with indices in $\mathcal J$, and $ (\mathcal F(A;\boldsymbol{\alpha},K),\mathbb{I}_n^{\mathcal J})$  describes the system \eqref{fracDyn}-\eqref{optProbP2}.

Notice that  given $K,K'$ where $K'>K$, we may have two different solutions to~\eqref{fracDyn} with $\mathcal F(A;\boldsymbol{\alpha},K)$ and $\mathcal F(A;\boldsymbol{\alpha},K')$ for an arbitrary $k$.

\section{PRELIMINARIES AND TERMINOLOGY}\label{prelim}

In this section, we review  some notions of observability to discrete-time fractional-order systems, and their counterpart using structural systems theory~\cite{dionSurvey}.

We start by recalling that a solution in closed-form to~\eqref{fracDyn} can be determined and given as follows.

\begin{lemma}[\cite{fracOrderdiscrete,fracOrderdiscreteJournal}]
The solution to~\eqref{fracDyn} is given as follows:
\begin{equation}
x_{k+1}=G_{k+1}x_{0},
\label{fracDynSol}
\end{equation}

\noindent where

\[
G_k=\left\{\begin{array}{cl}
A & \text{for } k=0,\\
\sum\limits_{j=0}^{k-1}A_jG_{k-1-j}& \text{for } k\ge 1.
\end{array}\right.
\]

\hfill $\diamond$
\end{lemma}

The notion of \emph{observability} for  discrete-time fractional-order systems is revisited in the next definition.

\begin{definition}
The   fractional-order system modeled by \eqref{fracDyn} and   \eqref{outputDiscrete} is \emph{observable} at time $k=0$ if and only if there exists some $K>0$ such that the state $x_0$ at time $k=0$ can be uniquely determined from the knowledge of the measured output $y_k$ with $k=1,\ldots,K$. \hfill $\diamond$
\end{definition}

Of a particular interest is the notion of the \emph{observability matrix} that is related with the system's observability.

\begin{definition}[\cite{fracOrderdiscrete,fracOrderdiscreteJournal}]
The observability matrix  associated with the discrete-time fractional order dynamics for a given time $k$ as described in~\eqref{fracDyn}, and the measured output 

\begin{equation}
y_k=Cx_k,
\label{outputDiscreteFull}
\end{equation}

\noindent is given as follows:
\[
\mathcal O_k^f=[(CG_0)^{\intercal} \ (CG_1)^{\intercal} \ \ldots \ (CG_{k-1})^{\intercal}]^{\intercal}.
\]

\hfill $\diamond$
\end{definition}

Additionally, we have the following result.

\begin{theorem}
The system described by \eqref{fracDyn}-\eqref{outputDiscreteFull}  is observable if and only if there exists a finite time $K$ such that $\text{rank } (\mathcal O_K^f)=n$. \hfill $\diamond$
\label{observabilityDynFrac}
\end{theorem}

Given that the precise numerical values of the network parameters are generally not available for the large-scale systems of interest, a natural direction is to consider structured systems~\cite{dionSurveyKyb} based reformulations of the above topology design problems, which we pursue in this paper. Representative work in structured systems theory may be found in \cite{Lin_1974,largeScale,Reinschke:1988,Murota:2009:MMS:1822520}, see also the survey \cite{dionSurvey} and references therein. The main idea is to reformulate and  study  an equivalent class of systems for which system-theoretic properties are investigated based on  the location of zeroes/non-zeroes of the state space representation matrices. Properties such as observability, in this framework, referred  as \emph{structural observability}, initially formalized for linear time-invariant systems~\cite{Lin_1974}. In fact, structural observability is a \emph{generic} property, i.e., almost all (with respect to the Lebesgue measure) realizations satisfying a given structure are observable~\cite{Reinschke:1988}. Given the similarity between observability criteria of the linear-time invariant and the discrete-time fractional order dynamics, one can readily extend the notion of structural observability to the latter. More precisely,  a pair $(\mathcal F(A;\boldsymbol{\alpha},K),C)$ is said to be structurally observable if there exists a pair $(\mathcal F(A';\boldsymbol{\alpha},K),C')$ with the same structure as $(\mathcal F(A;\boldsymbol{\alpha},K),C)$, i.e., same locations of zeroes and non-zeroes, such that $(\mathcal F(A';\boldsymbol{\alpha},K),C')$ is observable. By density arguments \cite{Reinschke:1988}, it may be shown that if a pair $(\mathcal F(A;\boldsymbol{\alpha},K),C)$ is structurally observable, then almost all (with respect to the Lebesgue measure) pairs with the same structure as $(\mathcal F(A;\boldsymbol{\alpha},K),C)$ are observable. In essence, structural observability is a property of the structure of the pair $(\mathcal F(A;\boldsymbol{\alpha},K),C)$ and not the specific numerical values.

Now, we associate with each subsystem a directed graph (digraph) $\mathcal D\equiv\mathcal D(A,\mathbb{I}^{\mathcal J}_n)=(\mathcal V,\mathcal E)$, referred to as \emph{system digraph}, with vertex set $\mathcal V$ and edge set $\mathcal E$, where $\mathcal V=\mathcal X\cup \mathcal Y$ with $\mathcal X=\{x^1,\ldots, x^n\}$ and $\mathcal Y=\{y^1,\ldots, y^{|\mathcal J|}\}$  represents the \emph{state} and \emph{output vertices}, respectively. In addition,  $\mathcal E=\mathcal E_{\mathcal X,\mathcal X}\cup \mathcal E_{\mathcal X,\mathcal Y}$ where $\mathcal E_{\mathcal X,\mathcal X}=\{(x^{k},x^j): A_{jk}\neq 0\}$ and   $\mathcal E_{\mathcal X,\mathcal Y}=\{(x^{j},y^j):  j\in \mathcal J \}$ represents the \emph{state edges} and \emph{output edges}, respectively. Similarly, we can define a \emph{state digraph} $\mathcal D(A)=(\mathcal X,\mathcal E_{\mathcal X,\mathcal X})$. Later, the matrix $A$ is given in terms of $A=A_1\lor \ldots \lor A_m$, where $\lor$ corresponds to the entry-wise  operation where if at least one of the entries is non-zero, then it provides a non-zero entry, and zero otherwise. A state vertex is said to be \emph{non-accessible} by an output vertex if there exists no directed path, i.e., a sequence of directed edges where every edge ends in a vertex that is starting of another edges and no vertex is used twice,  the state vertex to any output vertex. Additionally, we need to introduce the notion of a bipartite graph $\mathcal B(M)$ associated with a $m_1\times m_2$ matrix $M$ given by  $\mathcal B(M)=(\mathcal R,\mathcal C,\mathcal E_{\mathcal C,\mathcal R})$, where $\mathcal R=\{r_1,\ldots,r_{m_1}\}$ and $\mathcal C=\{c_1,\ldots,c_{m_2}\}$ correspond to the labeling row vertices and column vertices, respectively; further, $\mathcal E_{\mathcal C,\mathcal R}=\{(c_{j},r_i): M_{ij}\neq 0\}$. The bipartite graph is  an undirected graph with vertex set given by the union of the partition sets $\mathcal C$ and $\mathcal R$, which we refer to as left and right vertex sets, respectively. A matching $M\subset \mathcal E_{\mathcal C,\mathcal R}$ is a collection of edges that have no vertices in common. A maximum matching is a matching with maximum cardinality among all possible matchings. For ease of reference, if a vertex in the left and right vertex set does not belong to an edge in a maximum matching we refer to as a  right- and left-unmatched vertex, respectively. In addition, we can consider weights associated with the edges in a bipartite graph, so  we can consider the problem of determining the maximum matching with the minimum sum of the weights, that we refer to as the \emph{minimum weight maximum matching}.  A digraph $\mathcal{D}_S=(\mathcal V_S,\mathcal E_S)$ is a \emph{subgraph} of $\mathcal{D}=(\mathcal V,\mathcal E)$ if $\mathcal V_S\subseteq \mathcal V$ and $\mathcal E_S\subseteq \mathcal E$. Finally, a \emph{strongly connected component} (SCC) is a maximal subgraph (there is no other subgraph, containing it, with the same property) $\mathcal{D}_S=(\mathcal V_S,\mathcal E_S)$ of $\mathcal{D}$ such that for every $u,v \in \mathcal V_S$ there exists a path from $u$ to $v$ and from $v$ to $u$. 
We can create a \textit{directed acyclic graph} (DAG) by visualizing each SCC as a virtual node, where there is a directed edge between vertices belonging to two SCCs if and only if there exists a directed edge connecting the corresponding SCCs in the digraph $\mathcal D=(\mathcal V,\mathcal E)$, the original digraph. The DAG associated with $\mathcal{D}(A)$ can be computed efficiently in $\mathcal{O}(|\mathcal V|+|\mathcal E|)$~\cite{Cormen}. The SCCs in the DAG may be further categorized as follows.

\begin{definition}\label{linkedSCC}
\cite{PequitoJournal} An SCC is said to be linked if it has at least one incoming/outgoing edge from another SCC. In particular, an SCC is \textit{non-bottom linked} if it has no outgoing edges from its vertices to the vertices of another SCC.
\hfill $\diamond$
\end{definition}

Finally,  consider a $m_1\times m_2$ matrix $M$,  and let $\mathbb{M}=\{P\in\mathbb{R}^{m_1\times m_2}: \  P_{ij}=0 \text{ if } M_{ij}=0\}$, then the \emph{generic rank} (g-rank) of $M$ is given by
$
\text{g-rank}(M)=\max\limits_{P\in\mathbb{M}}\text{ rank}(P).
$

Now, we revise the structural observability necessary and sufficient conditions for linear switching systems. 

\begin{theorem}[\cite{Liu20133531}]
Consider a linear continuous-time switching system

\begin{equation}
\dot x(t)=A_{\sigma(t)}x(t),
\label{timeVarSys}
\end{equation}
where $\sigma: \mathbb{R}^+\rightarrow  M\equiv \{1,\ldots, m\}$ is a switching signal, and $x(t)\in\mathbb{R}^n$ the state of the system at the instance of time $t$.  In addition, let the measured output to be given by 
\begin{equation}
y(t)=Cx(t).
\label{outputContinuous}
\end{equation}
The linear continuous-time switching system \eqref{timeVarSys}-\eqref{outputContinuous} is structurally observable if and only if the following two conditions hold:
\begin{enumerate}
\item[(i)] $\mathcal D(A_1\lor \ldots \lor A_m, \mathbb{I}^{\mathcal J}_n)$ has no non-accessible state vertex;
\item[(ii)] $\text{g-rank}\left([ A_1 , \ldots, A_m,\mathbb{I}_n^{\mathcal J}]\right)=n$. \hfill $\diamond$ 
\end{enumerate}
\label{necSufCondStructSwitching}
\end{theorem}

In addition, consider the following intermediate results.

 \begin{lemma}[\cite{PequitoJournal}]
 $\mathcal D(A_1\lor \ldots \lor A_m, \mathbb{I}^{\mathcal J}_n)$ has no non-accessible state vertex if and only if there exits an edge  to an output vertex in  $\mathcal D(A_1\lor \ldots \lor A_m, \mathbb{I}^{\mathcal J}_n)$ from a state vertex in each non-bottom linked SCC of the DAG associated with  $\mathcal D(A_1\lor \ldots \lor A_m)$. \hfill $\diamond$ 
 \label{lemma2}
\end{lemma}

  \begin{lemma}[\cite{dionSurveyKyb}]
There exists a maximum matching of $\mathcal B(M)$ with size $n$ if and only if $\text{g-rank}(M)=n$.\hfill $\diamond$
\label{lemma1} 
\end{lemma}

\section{Main Results}\label{mainresults}

In this section, we present the main results of this paper. More precisely, we provide the solution to $\mathcal P_1$ that can be determined using Algorithm~\ref{mainAlgorithmNew}. Algorithm~\ref{mainAlgorithmNew} is shown to be correct and with polynomial complexity in Theorem~\ref{TheoremCorrectness}. 

First, we provide necessary and sufficient conditions to ensure structural observability of 
discrete time fractional-order dynamical systems.

\begin{theorem}
A discrete time fractional-order dynamical system~\eqref{fracDyn}-\eqref{outputDiscrete} is structurally observable at time $K$ if and only if the following two conditions hold:
\begin{enumerate}
\item[(i)] $\mathcal D(G_0\lor G_1\lor \ldots \lor G_K, \mathbb{I}^{\mathcal J}_n)$ has no non-accessible state vertex;
\item[(ii)] $\text{g-rank}\left([G_0 , G_1,\ldots, G_K,\mathbb{I}_n^{\mathcal J}]\right)=n$,
\end{enumerate}
where $G_i$ is as described in~\eqref{fracDynSol}.
\hfill $\diamond$ 
\label{fracToSwitch}
\end{theorem}

\begin{proof}
First, we notice that a  linear continuous-time switching system \eqref{timeVarSys}-\eqref{outputContinuous} is structurally observable if the  observability matrix  associated with it given by
\begin{align}\notag
&\mathcal O_k^s=[C^{\intercal} \ (CA_1)^{\intercal} \ \ldots \ (CA_m)^{\intercal} \ (CA_1A_2)^{\intercal} \ (CA_2A_1)^{\intercal}  \\ 
 & \ldots  (CA_{\rho(1)}^{i_1}\ldots A_{\rho(m)}^{i_m})^{\intercal}   \ldots (CA_1^{n-1})^{\intercal} \ldots (CA_m^{n-1})^{\intercal} ]^{\intercal},
\notag \end{align}
has generic rank equal to $n$. From an algebraic point of view, if we set $A_{\sigma (i)}=G_{i-1}$ (with $i=1,\ldots,K+1$) and $C=\mathbb{I}^{\mathcal J}_n$. Then, following the same steps as in Theorem 4 in~\cite{Liu20133531}, we obtain that the generic rank of $\mathcal O^{s}_{K+1}$ equals  that of $\mathcal O^{f}_{K}$. Consequently, conditions (i)-(ii) follow by invoking Theorem~\ref{necSufCondStructSwitching}.\end{proof}

As a direct consequence of Theorem~\ref{fracToSwitch}, by invoking Lemma~\ref{lemma1} and Lemma~\ref{lemma2}, we obtain the following result.

 \begin{corollary}
A fractional-order dynamical system~\eqref{fracDyn}-\eqref{outputDiscrete} is structurally observable at time $K$ if and only if the following two conditions hold:
\begin{enumerate}
\item[(i)] there exits an edge to an output vertex in $\mathcal D(G_0\lor G_1\lor \ldots \lor G_K, \mathbb{I}^{\mathcal J}_n)$ from a state vertex in each non-bottom linked SCC of the DAG associated with  $\mathcal D(G_0\lor G_1\lor \ldots \lor G_K)$;
\item[(ii)] there exists a maximum matching of $\mathcal B([ G_0\lor G_1\lor \ldots \lor G_K,\mathbb{I}_n^{\mathcal J}])$ with size $n$,
\end{enumerate}
where $G_i$ is as described in~\eqref{fracDynSol}.  \hfill $\diamond$ 
\label{CorollaryfracToSwitch}
 \end{corollary}


\begin{algorithm}[h]
\small
\KwIn{$G_i, \text{ with } i=1,\ldots,K $}
\KwOut{Minimal  dedicated output placement $\mathbb{I}^{\mathcal J}_n$}

$\mathbf{Step  \ 1.}$ Determine the  non-bottom linked SCCs  $\mathcal{N}^T_i$, $i\in \mathcal I\equiv\{1,\cdots, \beta\}$, of $\mathcal D(G_0\lor \ldots \lor G_K)=(\mathcal X,\mathcal E_{\mathcal X,\mathcal X})$ 

$\mathbf{Step  \ 2.}$ Consider a weighted bipartite graph $\mathcal B([ G_0^\intercal , \ldots, G_K^\intercal,S])=(\mathcal R,\mathcal C,\mathcal E_{\mathcal R,\mathcal C})$, where $S$ is a $n\times \beta$ matrix and $S_{i,j}=1$ if $x_i \in \mathcal N^T_j$, and  the column vertices be re-labeled as follows: the columns of $A_i$ are indexed by $\{c^i_1,\ldots, c^i_n\}$, and the columns of $S$ are indexed by $\{s_1,\ldots, s_{\beta}\}$. In addition, let the weight of the edges  $e\in \mathcal R\times\left(\bigcup\limits_{i=1,\ldots,m} \{c^i_1,\ldots, c^i_n\}\right)$ be equal to zero, the weight on the edges  $e\in \mathcal R\times\{s_1,\ldots, s_{\beta}\}$ be equal to one, and all non-existing edges corresponding to infinite weight.

$\mathbf{Step  \ 3.}$ Let $M'$ be the maximum matching incurring in the minimum cost of the weighted bipartite graph presented in Step 2. 

$\mathbf{Step  \ 4.}$ Take $\mathcal J'=\{i: (.,s_i)\in M'\}$, i.e., the column vertices of $S$  that belong to the edges in the  MWMM $M'$ (i.e., those with weight one). In addition, let $\mathcal J''=\{1,\ldots,n\}\setminus \{k: (.,c_j^k)\in M', j=1,\ldots, n\}$, and $\mathcal J'''= \bigcup\limits_{p\in \{1,\ldots,\beta\}\setminus \mathcal J'}\min \{j:  x_j\in \mathcal N^T_p\}$.

$\mathbf{Step  \ 5.}$ Set $\mathcal J=\mathcal J'\cup\mathcal J''\cup \mathcal J'''$.

\caption{Computing a minimal dedicated output placement ensuring structural observability}
\label{mainAlgorithmNew}
\end{algorithm}

Now, we   provide an efficient algorithmic procedure to compute a solution to $\mathcal P_1$,  described in Algorithm~\ref{mainAlgorithmNew}, that ensures that both conditions in Corollary~\ref{CorollaryfracToSwitch} are satisfied.  Briefly, Algorithm~\ref{mainAlgorithmNew} consists in finding a minimum weight maximum matching (MWMM) of a bipartite graph  $\mathcal B([ G_0^\intercal , \ldots, G_K^\intercal,S])$, where the matrix $S$ has as many columns as the number of non-bottom linked SCCs in the DAG representation of $\mathcal D(G_0\lor \ldots \lor G_K)$, and the non-zero entries in  column $i$ of $S$ correspond to the indices of the state variables that belong to the $i$-th non-bottom linked SCC. In addition, we consider weights in the edges of the bipartite graph: those associated with the nonzero entries of $A_i$ have zero weight, unitary weight is considered to the nonzero entries in $S$, and infinite weight otherwise.

Therefore, if an edge with unitary weight belongs to the MWMM, then it contains a row vertex, which implies that dedicated sensors need to be assigned to the state variable with  the index of the row vertex, whereas all the remaining unitary edges not in the MWMM mean that these cannot be used to further increase the cardinality of the matching. Yet,  notice that $\mathcal B([G_0^\intercal , \ldots, G_K^\intercal,\mathbb{I}^{\mathcal J'}_n])$ may not have a maximum matching of size $n$; hence, additional dedicated sensors indexed by $\mathcal J''$  need to be considered, corresponding to the indices of the row vertices that are unmatched (i.e., the left-unmatched vertices of $\mathcal B([ G_0^\intercal , \ldots, G_K^\intercal,\mathbb{I}^{\mathcal J'}_n])$ associated with the MWMM). In summary, the indices in $\mathcal J'$ correspond to the sensors that increase the g-rank and minimize the number of non-accessible state vertices, since they are assigned to non-bottom linked SCCs of $\mathcal D(G_0\lor \ldots \lor G_K)$, whereas $\mathcal J''$ ensures that the g-rank is equal to $n$. Thus, the last step consists in considering  an additional subset of dedicated sensors $\mathcal J'''$  to ensure that there are no non-accessible state vertices in $\mathcal D(G_0\lor \ldots \lor G_K, \mathbb{I}^{\mathcal J}_n)$, where $\mathcal J=\mathcal J'\cup\mathcal J''\cup \mathcal J'''$; in  other words, $\mathcal J'''$ consists of indices of the minimum  dedicated sensors  assigned to a state variable in each non-bottom linked SCC that contains non-accessible state vertices.

The  next result establishes  the correctness and analyzes the implementation complexity  of Algorithm~\ref{mainAlgorithmNew}.

\begin{theorem}
Algorithm~\ref{mainAlgorithmNew} is correct, i.e., it provides a solution to $\mathcal P_1$. Furthermore, its computational complexity is  $\mathcal O(|\mathcal C|^{3})$, where $|\mathcal C|$ denotes the number of column vertices in $\mathcal B([G_0^\intercal , \ldots, G_K^\intercal,S])$. 
\hfill $\diamond$
\label{TheoremCorrectness}
\end{theorem}

\begin{proof}
The correctness of Algorithm~\ref{mainAlgorithmNew} follows from noticing that the indices in $\mathcal J'$ identifies the minimum set of dedicated outputs that  simultaneously  maximizes the increase in the g-rank of   $[G_0^\intercal , \ldots, G_K^\intercal,\mathbb{I}_n^{\mathcal J'}]$ with respect to $[G_0^\intercal , \ldots, G_K^\intercal]$ by $|\mathcal J'|$,  and  the dedicated outputs assigned to state variables in different non-bottom linked SCCs.  
This follows from observing  that (by construction) $\mathcal B([G_0^\intercal , \ldots, G_K^\intercal])$ incurs in a minimum weight maximum matching $M$ with zero weight and size $|M|$. From Lemma~\ref{lemma1}, it follows that $\text{g-rank}([G_0^\intercal , \ldots, G_K^\intercal])=|M|$. Subsequently, a minimum weight maximum matching $M'$ of $\mathcal B([G_0^\intercal , \ldots, G_K^\intercal, S])$ equals $|M'|-|M|$; hence, increasing by $|M'|-|M|$ the g-rank of  $[G_0^\intercal , \ldots, G_K^\intercal,\mathbb{I}_n^{\mathcal J'}]$ with respect to $[G_0^\intercal , \ldots, G_K^\intercal]$, and contributing to satisfy condition (ii) in Corollary~\ref{CorollaryfracToSwitch} -- but may not be enough yet to ensure this condition, which is accounted for with the set $\mathcal J''$. In addition, by construction of $S$ it follows that $\mathbb{I}_n^{\mathcal J'}$ correspond to dedicated outputs that are assigned to state variables in different non-bottom linked SCCs; hence, $|\mathcal J'|$ non-bottom linked SCCs have outgoing edges from its state variables in different outputs in the system digraph; thus,  contributing to satisfy condition (i) in Corollary~\ref{CorollaryfracToSwitch} -- but may not be enough yet to ensure this condition, which is accounted for with the set $\mathcal J'''$.  Subsequently, the number of total number of additional dedicated outputs $\mathbb{I}_n^{\mathcal J''}$ required  such that $\text{g-rank}\left([G_0^\intercal , \ldots, G_K^\intercal,\mathbb{I}_n^{\mathcal J'},\mathbb{I}_n^{\mathcal J''}]\right)=n$ is minimized by considering Step~4. Similarly, the number of total number of additional dedicated outputs $\mathbb{I}_n^{\mathcal J'''}$ required  such that there exist no non-accessible state vertices in $\mathcal D(G_0 \lor \ldots \lor G_K, \mathbb{I}_n^{\mathcal J'\cup \mathcal J'''})$ is minimized by considering Step~4. Notice that $\mathbb{I}_n^{\mathcal J''}$ are not assigned to non-bottom linked SCCs, otherwise they would have been  considered in $\mathbb{I}_n^{\mathcal J'}$. Therefore, by setting $\mathcal J=\mathcal J'\cup\mathcal J''\cup\mathcal J'''$, as in Step~5, we obtain a solution to~$\mathcal P_1$, since  Corollary~\ref{CorollaryfracToSwitch} yields.

The computational complexity follows from noticing that Step 2 can be solved using the Hungarian algorithm that finds a MWMM in $\mathcal O(\max \{|\mathcal C|,|\mathcal R|\}^3)$, whereas all other steps have linear complexity; hence, Step 2 dominates the final computational complexity, leading to the final complexity of $\mathcal O(|\mathcal C|^3)$ obtained, where we note that $|\mathcal C|\ge |\mathcal R|$.
\end{proof}

\section{AN ILLUSTRATIVE EXAMPLE}\label{illustrativeexample}

In what follows, we consider a brain activity dataset \cite{BrainData} describing the dynamics of the brain while the individuals perform different motor and imagery tasks. The brain activity was measured using the brain-computer interface and consists of 64-channel electroencephalogram (EEG). Each individual performed 14 experimental runs consisting of one minute with eyes open, one minute with eyes closed, and three two-minute runs of interacting sessions of opening/closing the corresponding left/right fist as a function of where the targets appear on screens.  To this end, a geodesic sensor net for EEG was considered  (illustrated in Figure~\ref{f:Helmet}), which sensor location schematics and enumeration is presented in Figure~\ref{f:HelmetSensors}.


\begin{figure}[h!]
\centerline{\includegraphics[scale=0.4]{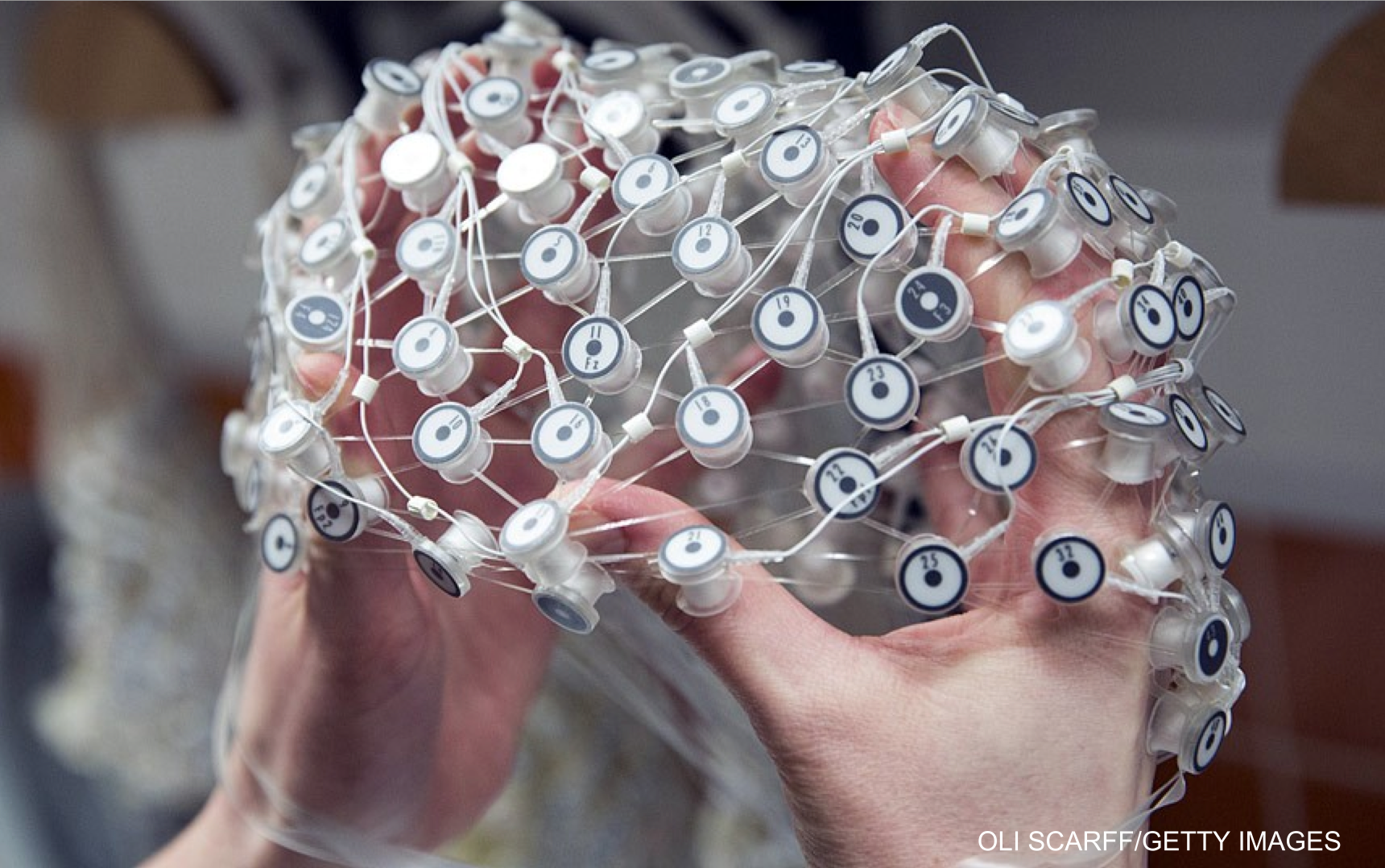}}
\caption{64-channel geodesic sensor net for an EEG.}
\label{f:Helmet}
\end{figure}

\begin{figure}[h!]
\centerline{\includegraphics[scale=0.5]{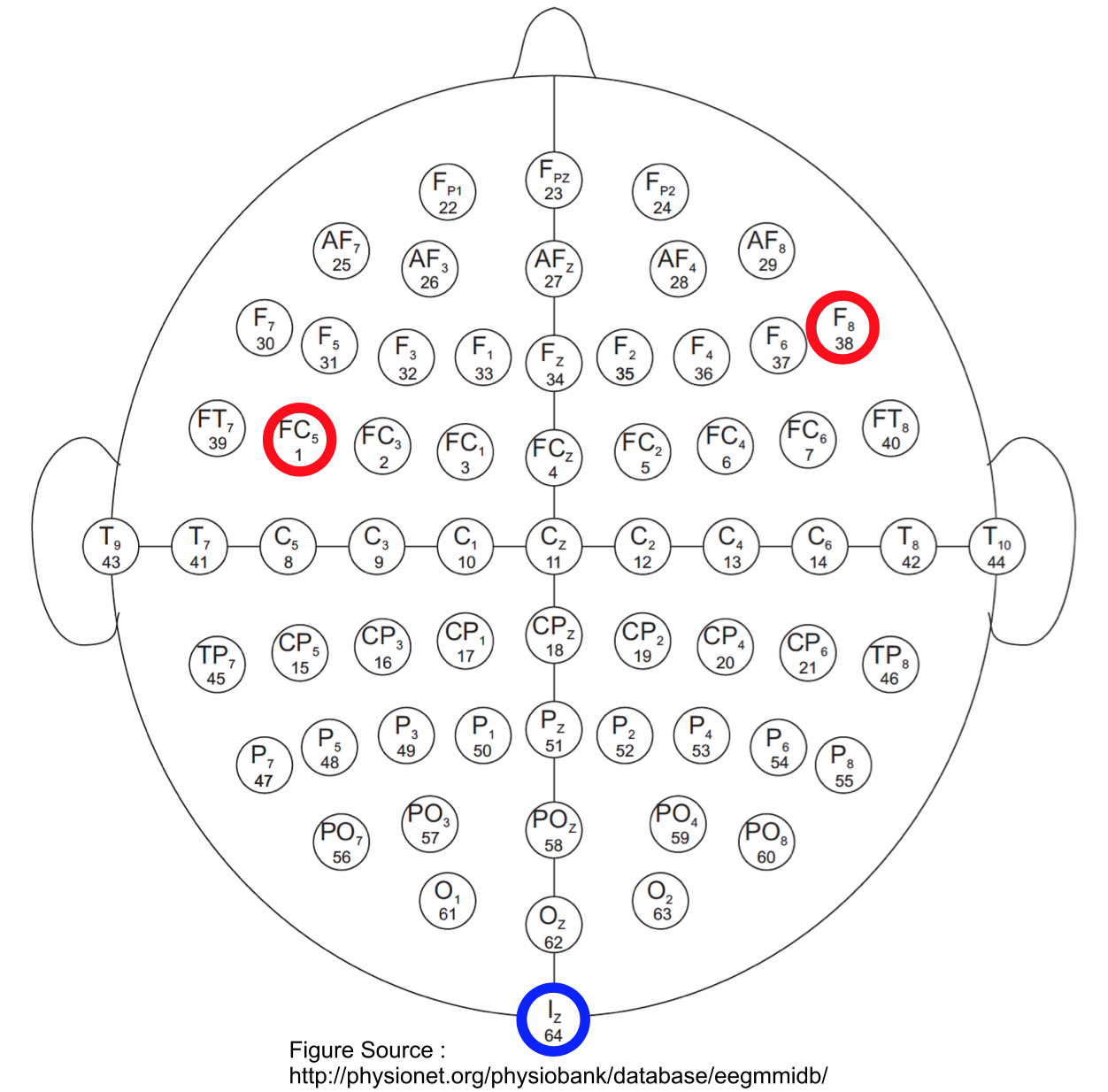}}
\caption{Sensor distribution in the 64-channel geodesic sensor net for an EEG depicted in Figure~\ref{f:Helmet}. The sensor in blue represents the sensor whose neuronal activity is simulated using the fractional order system identified, and compared with data recorded in Figure~\ref{f:sim}. The sensors in red represent the minimum number of sensors required to ensure structural observability of a fractional order system identified, and considering a sparsity of $80\%$.}
\label{f:HelmetSensors}
\end{figure}

Using the wavelet technique described in \cite{Achard}, we estimated the parameters $\alpha_{j}$ ($j=1,\ldots,64$) in~\eqref{fracDyn}, corresponding to the fractional order derivatives governing the dynamics of the monitoring neuronal regions and the coupling matrix $A$ that encapsulates the spatial fractal connectivity among different brain regions. To give some intuition of how spread the parameters $\alpha_{j}$ corresponding to the fractional order derivatives are, the estimated values range between $0.97$ and $1.28$. Hence, providing evidence of the fractional order dynamics; more precisely, some brain regions exhibit pronounced long-range memory dynamics which can be better captured by fractional order state space representation.   In Figure~\ref{f:sim}, we contrast  the recorded and simulated data on channel 64 (depicted in blue in Figure~\ref{f:HelmetSensors}), during the first two-minutes run of interacting sessions of opening/closing the corresponding left/right fist as a function of where the targets appear on screens. 


\begin{figure}[h!]
\centerline{\includegraphics[width=5.5in,keepaspectratio]{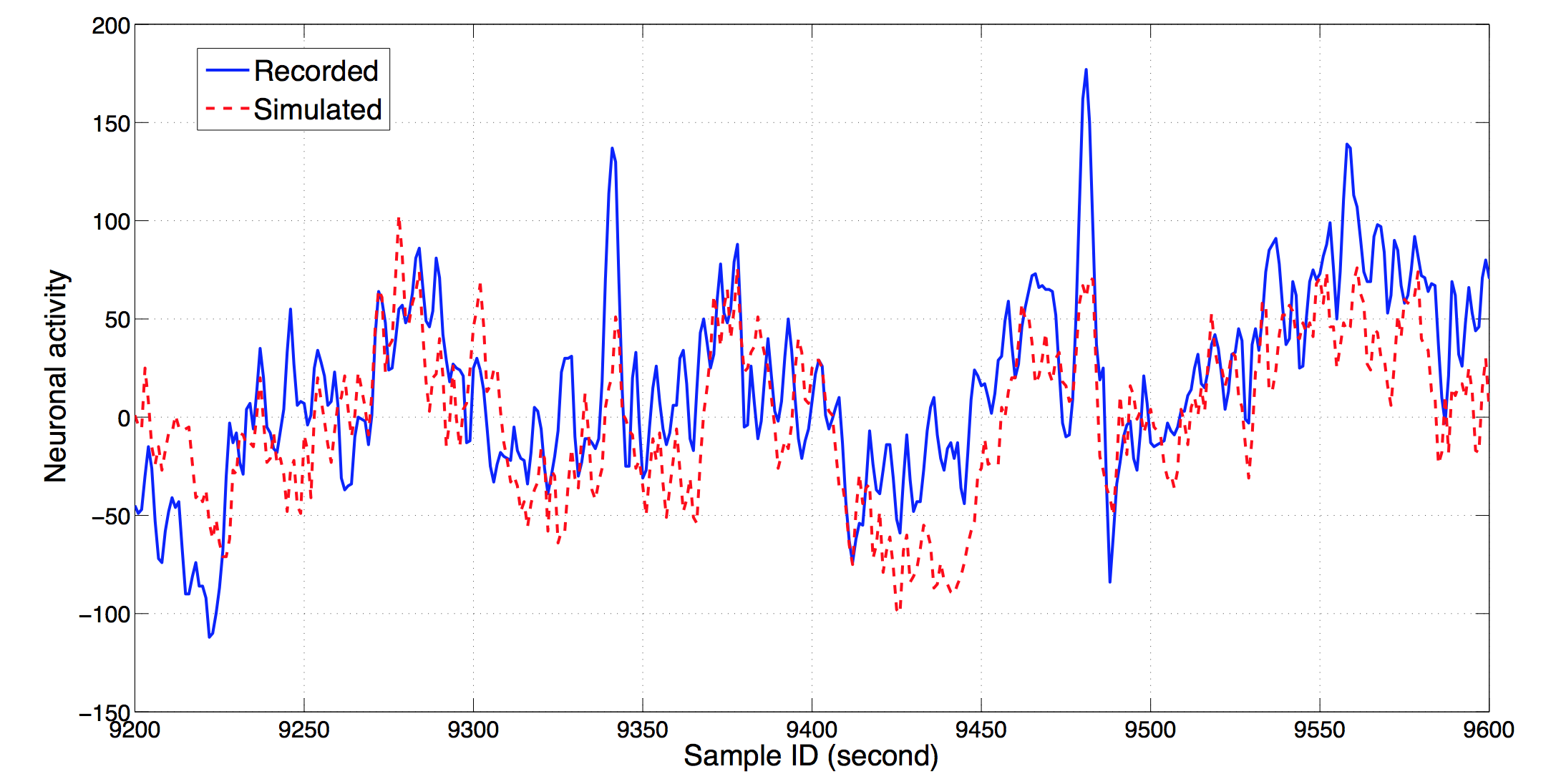}}
\caption{In this figure, we show the recorded and simulated data on channel 64 (depicted in blue in Figure~\ref{f:HelmetSensors}), where a sample is taken every second, during the first two-minutes run of interacting sessions of opening/closing the corresponding left/right fist as a function of where the targets appear on screens.  }
\label{f:sim}
\end{figure}

\begin{figure}[h!]
\centerline{\includegraphics[width=5.5in,keepaspectratio]{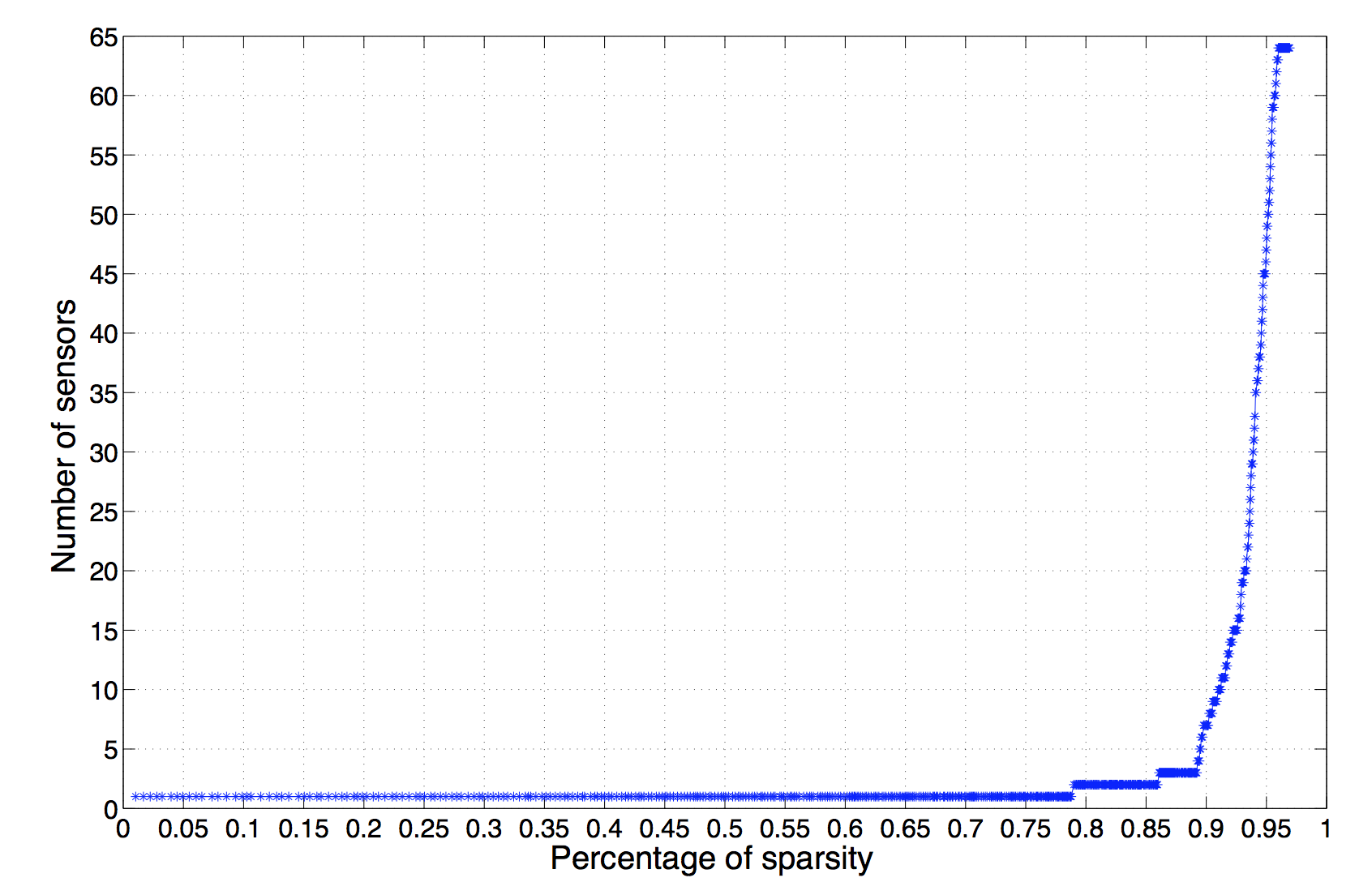}}
\caption{Number of required sensors to ensure structural observability as a function of the sparsity.}
\label{f:Sparsity}
\end{figure}

A common approach in determining the brain connectivity is to consider the correlations between different measured signals~\cite{BassettMethods}. In contrast, we advocate for a structural approach of the signal dynamics induced by the sparsification of the coupling matrix $A$. This structural analysis offers a robust mathematical framework for capturing both the spatial and temporal interactions among distant brain regions and identifying the minimum number of sensors that provide full observable description of brain dynamics. Following the theoretical results and analysis outlined in Corollary 2, by resorting to Algorithm~1, we determine the optimal number of sensors as a function of the sparsity present in the coupling matrix $A$ for a time horizon equal to the dimension of the state space. Figure~\ref{f:Sparsity} shows that the higher the sparsity of the coupling matrix $A$ of the brain dynamics is, the more sensors are required for observability purposes. An important observation one could make is that for the case when the system connectivity is within the range of $ 20 - 30 \%$ (which implies that the sparsity is within $70 - 80 \%$ interval), which is characteristic to structural/functional brain connectivity~\cite{Jirsa2007,Sporns:2010:NB:2024602,Sporns}, the required number of sensors exhibit a {\it first order phase transition}. Simply speaking, to monitor the brain activity when the connectivity is between $70 \%$ and $80 \%$ we only need 2 sensors, depicted in red in Figure~\ref{f:HelmetSensors}. However, when the connectivity is just $10 \%$ we are required to use 7 sensors and the required number of sensors spikes fast with increasing sparsity.

\section{CONCLUSIONS AND FURTHER RESEARCH}\label{conclusions}

In this paper, motivated by the problem of minimum placement of probes  in the brain to obtain dynamic observability,  we  presented methodologies to determine the minimum dedicated sensor placement for discrete time fractional-order systems. Further, the results presented can be used to determine the minimum dedicated input (i.e., inputs that actuate single state variables) placement for aforementioned  systems by invoking the duality between controllability and observability in the case of  discrete-time fractional-order systems. Finally, it would be interesting to generalize the present results for a larger class of systems, and validate the proposed methodology in different brain regions and using different technologies.

\footnotesize 


\bibliographystyle{IEEEtran}

\bibliography{IEEEabrv,cdc2015}

\begin{thebibliography}{10}
\providecommand{\url}[1]{#1}
\csname url@samestyle\endcsname
\providecommand{\newblock}{\relax}
\providecommand{\bibinfo}[2]{#2}
\providecommand{\BIBentrySTDinterwordspacing}{\spaceskip=0pt\relax}
\providecommand{\BIBentryALTinterwordstretchfactor}{4}
\providecommand{\BIBentryALTinterwordspacing}{\spaceskip=\fontdimen2\font plus
\BIBentryALTinterwordstretchfactor\fontdimen3\font minus
  \fontdimen4\font\relax}
\providecommand{\BIBforeignlanguage}[2]{{%
\expandafter\ifx\csname l@#1\endcsname\relax
\typeout{** WARNING: IEEEtran.bst: No hyphenation pattern has been}%
\typeout{** loaded for the language `#1'. Using the pattern for}%
\typeout{** the default language instead.}%
\else
\language=\csname l@#1\endcsname
\fi
#2}}
\providecommand{\BIBdecl}{\relax}
\BIBdecl

\bibitem{NeuralNetControllingNodes}
Y.~Tang, H.~Gao, W.~Zou, and J.~Kurths, ``Identifying controlling nodes in
  neuronal networks in different scales,'' \emph{PLoS ONE}, vol.~7, no.~7, p.
  e41375, 07 2012.

\bibitem{controlBrainNet2014}
S.~{Gu}, F.~{Pasqualetti}, M.~{Cieslak}, S.~T. {Grafton}, and D.~S. {Bassett},
  ``{Controllability of Brain Networks},'' \emph{ArXiv e-prints}, Jun. 2014.

\bibitem{FrancisModeling}
A.~{Franci}, G.~{Drion}, and R.~{Sepulchre}, ``{Modeling the modulation of
  neuronal bursting: a singularity theory approach},'' \emph{ArXiv e-prints},
  May 2013.

\bibitem{FranciSIAM}
A.~Franci, G.~Drion, and R.~Sepulchre, ``An organizing center in a planar model
  of neuronal excitability,'' \emph{SIAM Journal on Applied Dynamical Systems},
  vol.~11, no.~4, pp. 1698--1722, 2012.

\bibitem{RuthTAC13}
J.-S. Li, I.~Dasanayake, and J.~Ruths, ``Control and synchronization of neuron
  ensembles,'' \emph{IEEE Transactions on Automatic Control}, vol.~58, no.~8,
  pp. 1919--1930, 2013.

\bibitem{Franci2013}
\BIBentryALTinterwordspacing
A.~Franci, G.~Drion, V.~Seutin, and R.~Sepulchre, ``{A Balance Equation
  Determines a Switch in Neuronal Excitability},'' Mar. 2013. [Online].
  Available: \url{http://arxiv.org/abs/1209.6445}
\BIBentrySTDinterwordspacing

\bibitem{Drion}
G.~Drion, A.~Franci, V.~Seutin, and R.~Sepulchre, ``A novel phase portrait for
  neuronal excitability,'' \emph{PLoS ONE}, vol.~7, no.~8, p. e41806, 08 2012.

\bibitem{Lundstrom}
B.~N. Lundstrom, M.~H. Higgs, W.~J. Spain, and A.~L. Fairhall, ``Fractional
  differentiation by neocortical pyramidal neurons,'' \emph{Nature
  neuroscience}, vol.~11, no.~11, pp. 1335--1342, 11 2008.

\bibitem{Ghorbani}
M.~Ghorbani and P.~Bogdan, ``A cyber-physical system approach to artificial
  pancreas design,'' in \emph{Proceedings of the Ninth IEEE/ACM/IFIP
  International Conference on Hardware/Software Codesign and System Synthesis},
  ser. CODES+ISSS '13.\hskip 1em plus 0.5em minus 0.4em\relax Piscataway, NJ,
  USA: IEEE Press, 2013, pp. 17:1--17:10.

\bibitem{B2012}
P.~Bogdan, R.~Marculescu, S.~Jain, and R.~Gavila, ``An optimal control approach
  to power management for multi-voltage and frequency islands multiprocessor
  platforms under highly variable workloads,'' in \emph{Sixth IEEE/ACM
  International Symposium on Networks on Chip (NoCS)}, May 2012, pp. 35--42.

\bibitem{B2011}
P.~Bogdan and R.~Marculescu, ``A fractional calculus approach to modeling
  fractal dynamic games,'' in \emph{50th IEEE Conference on Decision and
  Control and European Control Conference}, Dec 2011, pp. 255--260.

\bibitem{fracOrderdiscreteJournal}
S.~Guermah, S.~Djennoune, and M.~Bettayeb, ``Controllability and observability
  of linear discrete-time fractional-order systems.'' \emph{Applied Mathematics
  and Computer Science}, vol.~18, no.~2, pp. 213--222, 2008.

\bibitem{Kaslik}
E.~Kaslik and S.~Sivasundaram, ``Dynamics of fractional-order neural
  networks,'' in \emph{The 2011 International Joint Conference on Neural
  Networks (IJCNN)}, July 2011, pp. 611--618.

\bibitem{Zhou2008973}
S.~Zhou, H.~Li, and Z.~Zhu, ``Chaos control and synchronization in a fractional
  neuron network system,'' \emph{Chaos, Solitons \& Fractals}, vol.~36, no.~4,
  pp. 973 -- 984, 2008.

\bibitem{SpectrumFev15}
T.~Denison, M.~Morris, and F.~Sun, ``Building a bionic nervous system,''
  \emph{Spectrum, IEEE}, vol.~52, no.~2, pp. 32--39, February 2015.

\bibitem{PequitoJournal}
\BIBentryALTinterwordspacing
S.~Pequito, S.~Kar, and A.~Aguiar, ``A framework for structural input/output
  and control configuration selection of large-scale systems,'' \emph{Accepted
  to IEEE Transactions on Automatic Control}, 2013. [Online]. Available:
  \url{http://arxiv.org/pdf/1309.5868}
\BIBentrySTDinterwordspacing

\bibitem{Hihi09}
H.~Hihi, ``Structural controllability of switching linear systems,''
  \emph{Journal of Computers}, vol.~4, no.~12, 2009.

\bibitem{Liu20133531}
X.~Liu, H.~Lin, and B.~M. Chen, ``Structural controllability of switched linear
  systems,'' \emph{Automatica}, vol.~49, no.~12, pp. 3531 -- 3537, 2013.

\bibitem{allerton}
G.~Ramos, S.~Pequito, A.~P. Aguiar, J.~Ramos, and S.~Kar, ``A model checking
  framework for linear time invariant switching systems using structural
  systems analysis,'' in \emph{Proc. of 51th Annual Allerton Conference on
  Communication, Control, and Computing}.\hskip 1em plus 0.5em minus
  0.4em\relax IEEE Press, 2013.

\bibitem{PequitoACC}
S.~Pequito, S.~Kar, and A.~Aguiar, ``A structured systems approach for optimal
  actuator-sensor placement in linear time-invariant systems,'' in \emph{Proc.
  of the 2013 American Control Conference}, pp. 6108--6113.

\bibitem{Baleanu}
D.~Baleanu, J.~Machado, and A.~Luo, \emph{Fractional Dynamics and
  Control}.\hskip 1em plus 0.5em minus 0.4em\relax Springer, 2011.

\bibitem{dionSurvey}
J.-M. Dion, C.~Commault, and J.~V. der Woude, ``Generic properties and control
  of linear structured systems: a survey.'' \emph{Automatica}, pp. 1125--1144,
  2003.

\bibitem{fracOrderdiscrete}
M.~Bettayeb, S.~Djennoune, S.~Guermah, and M.~Ghanes, ``Structural properties
  of linear discrete-time fractional-order systems,'' \emph{Proceedings of the
  17th World Congress of the International Federation of Automatic Control,
  Seoul, Korea}, pp. 15\,262--15\,266, 2008.

\bibitem{dionSurveyKyb}
J.-M. Dion, C.~Commault, and J.~V. der Woude, ``Characterization of generic
  properties of linear structured systems for efficient computations,''
  \emph{Kybernetika}, vol.~38, pp. 503--520, 2002.

\bibitem{Lin_1974}
C.~Lin, ``Structural controllability,'' \emph{IEEE Transactions on Automatic
  Control}, no.~3, pp. 201--208, 1974.

\bibitem{largeScale}
D.~D. Siljak, \emph{Large-Scale Dynamic Systems: Stability and
  Structure}.\hskip 1em plus 0.5em minus 0.4em\relax Dover Publications, 2007.

\bibitem{Reinschke:1988}
K.~J. Reinschke, \emph{{Multivariable control: a graph theoretic approach}},
  ser. Lect. Notes in Control and Information Sciences.\hskip 1em plus 0.5em
  minus 0.4em\relax Springer-Verlag, 1988, vol. 108.

\bibitem{Murota:2009:MMS:1822520}
K.~Murota, \emph{Matrices and Matroids for Systems Analysis}, 1st~ed.\hskip 1em
  plus 0.5em minus 0.4em\relax Springer Publishing Company, Incorporated, 2009.

\bibitem{Cormen}
T.~H. Cormen, C.~Stein, R.~L. Rivest, and C.~E. Leiserson, \emph{Introduction
  to Algorithms}, 2nd~ed.\hskip 1em plus 0.5em minus 0.4em\relax McGraw-Hill
  Higher Education, 2001.

\bibitem{BrainData}
\BIBentryALTinterwordspacing
L.~A.~N. Goldberger, A. L.~Amaral, L.~Glass, J.~M. Hausdorff, P.~C. Ivanov,
  R.~Mark, J.~Mietus, G.~Moody, C.-K. Peng, and H.~E. Stanley, ``Physiobank,
  physiotoolkit, and physionet: Components of a new research resource for
  complex physiologic signals,'' \emph{Circulation}, vol. 101, pp. e215--e220,
  2000. [Online]. Available:
  \url{http://circ.ahajournals.org/cgi/content/full/101/23/e215}
\BIBentrySTDinterwordspacing

\bibitem{Achard}
S.~Achard, D.~S. Bassett, A.~Meyer-Lindenberg, and E.~Bullmore, ``Fractal
  connectivity of long-memory networks,'' \emph{Phys. Rev. E}, vol.~77, p.
  036104, Mar 2008.

\bibitem{BassettMethods}
D.~S. Bassett and M.-E. Lynall, ``Network methods to characterize brain
  structure and function,'' \emph{Cognitive Neurosciences: The Biology of the
  Mind (Fifth Edition)}, 2009.

\bibitem{Jirsa2007}
V.~K. Jirsa and A.~McIntosh, Eds., \emph{Handbook of Brain Connectivity},
  1st~ed., ser. Understanding Somplex Systems.\hskip 1em plus 0.5em minus
  0.4em\relax New York: Springer Berlin Heidelberg, December 2007.

\bibitem{Sporns:2010:NB:2024602}
O.~Sporns, \emph{Networks of the Brain}, 1st~ed.\hskip 1em plus 0.5em minus
  0.4em\relax The MIT Press, 2010.

\bibitem{Sporns}
------, ``Contributions and challenges for network models in cognitive
  neuroscience,'' \emph{Nature neuroscience}, vol.~17, pp. 652--660, 2014.

\end{thebibliography}

\end{document}